\documentclass[10pt]{article}

\usepackage[utf8]{inputenc}
\usepackage{graphicx}
\usepackage{amsthm}
\usepackage{amsmath}
\usepackage{amssymb}
\usepackage{amsfonts}
\usepackage{mathtools}
\usepackage{dsfont}
\usepackage{mathrsfs}
\usepackage{enumerate}
\usepackage[normalem]{ulem}
\usepackage{todonotes}
\usepackage{ifthen}
\usepackage{hyperref}

\newtheorem{proposition}{Proposition}

\newtheorem{lemma}[proposition]{Lemma}
\newtheorem{corollary}[proposition]{Corollary}
\newtheorem{theorem}[proposition]{Theorem}

\newtheorem{fact}[proposition]{Fact}

\newtheorem{definition}[proposition]{Definition}

\newcommand{\N}{\mathbb{N}}
\newcommand{\R}{\mathbb{R}}
\newcommand{\Z}{\mathbb{Z}}
\newcommand{\C}{\mathbb{C}}
\newcommand{\Q}{\mathbb{Q}}
\newcommand{\A}{\mathscr{A}}
\newcommand{\mH}{\mathscr{H}}

\newcommand{\B}{\mathscr{B}}

\newcommand{\abs}[1]{\left | #1 \right |}

\newcommand{\map}[3]{{#1} \colon {#2} \to {#3}}
\newcommand{\eps}{\varepsilon}
\renewcommand{\phi}{\varphi}
\newcommand{\Ho}{\mbox{$H_0(\hat\C\setminus \R)$}}
\newcommand{\f}[1]{\widetilde{#1}}
\newcommand{\Ar}{\A(\R)}
\newcommand{\norm}[1]{\left\lVert #1 \right\rVert}
\newcommand{\LL}[0]{\mathcal{L}}
\newcommand{\Tt}{(T_t)_{t\geq 0}}

\newcommand{\dy}{\mathrm{d}y}
\newcommand{\ds}{\mathrm{d}s}

\DeclareMathOperator{\Rea}{Re}
\DeclareMathOperator{\Ima}{Im}

\DeclareMathOperator{\ind}{ind}

\DeclareMathOperator{\Int}{Int}

\newcommand{\indH}{\underset{\R\subset U}{\ind} H(U)}

\begin{document}

\title{{Semigroups of Hadamard multipliers on the space of real analytic functions}}
\author{{Anna Goli\'nska}}

 \maketitle

\begin{abstract}
An operator $M$ acting on the space of real analytic functions $\Ar$ is called a multiplier if every monomial is its eigenvector. In this paper we state some results concerning the problem of generating strongly continuous semigroups by multipliers. In particular we show when the Euler differential operator of finite order is a generator and when it is not. 

\end{abstract}

 \footnotetext[1]{


Author is supported by the National Science Centre (Poland) grant no. 2013/10/A/ST1/00091.}

\section{Introduction} 

By $\Ar$ we will denote the space of real analytic functions with its natural inductive topology, i.e.
\[
\Ar = \ind_{\R\subset U}H(U),
\]
where $U$ runs over all complex neighbourhoods of $\R$ and $H(U)$ is equipped with the usual compact-open topology. The topology of $\Ar$ is complicated, but we will only need the following special case for convergent sequences:

\begin{fact}
  A sequence $(f_n)$ converges to $f$ in the topology of $\Ar$ if and only if all the functions $f_n$ and $f$ extend as holomorphic functions to a complex neighbourhood $U$ of $\R$ and $f_n \to f$ in $H(U)$. 
\end{fact}

Let $L(\Ar)$ be the space of all linear continuous operators on the space of real analytic functions $\Ar$ with the topology of uniform convergence on bounded sets of $\Ar$. We say that an operator $M\in L(\Ar)$ is a mutliplier, if every monomial is its eigenvector, i.e. 
\begin{align*}
M(x^n)=m_nx^n \; \text{for all $n\in\N$}.
\end{align*}
We call the sequence $(m_n)_{n\in\N}$ a multiplier sequence. Since monomials are linearly dense in $\Ar$ a multiplier is uniquely determined by its multiplier sequence. By $(M,(m_n))$ we will denote the multiplier $M$ with the multiplier sequence $(m_n)_{n\in\N}$. We denote by $M(\R)$ the space of all multipliers and equip it with the topology induced from $L(\Ar)$. The basic examples of multipiers are:
\begin{itemize}
\item
Euler differential operator
\[
 Ef(x)=xf'(x),
\]                                                                                                                                                                                                                                                                                                                                                                                                                         \item 
dilation operator
\[
D_af(x)=f(ax),
\]    
\item
Hardy operator
\[
 Hf(x)=\frac{1}{x}\int_0^xf(y)dy.
\]
\end{itemize}

For more information on multipliers on $\Ar$ we refer to \cite{DL1,DL2,DL}.

In this paper we consider semigroups generated by multipliers. 
Consider the abstract Cauchy problem
\begin{align}
 \begin{split}
 \frac{\partial}{\partial t} u(t)&= M u(t),\\
 u(0)&= f,
 \label{evolutionEq}
 \end{split}
\end{align}
where $M\in M(\R)$, $f\in \Ar$. A classical approach to solve (\ref{evolutionEq}) is to study if the operator $M$ generates a strongly continuous semigroup of bounded linear operators $\{T_t: t\geq 0\}$. In this paper we will try to answer the question: Which multipliers generate a strongly continuous semigroup? Note that on a non-Banach locally convex space, a continuous linear operator does not always generate a strongly continuous semigroup.

\section{Preliminaries}
In this section we will introduce some notation and recall basic facts from the general theory of semigroups (more details can be found in \cite{Komura}).

Let $X$ be a locally convex space and $(T_t)_{t\geq 0}$ a family of bounded operators on $X$. The family $(T_t)_{t\geq 0}$ is said to be a semigroup, if it satisfies the following conditions:
\begin{enumerate}[$(1)$]
\item $T_tT_s=T_{t+s}$ for all $t,s\geq 0$,
\item $T_0=I$ (the identity operator).
\end{enumerate} 
If in addition it satisifies
\begin{enumerate}[$(3)$]
\item $\lim_{t\rightarrow s}T_tx=T_sx$ for any $s\geq 0$ and any $x\in X$.
\end{enumerate}
then $(T_t)_{t\geq 0}$ is called a $C_0$-semigroup (or strongly continuous  semigroup).

If the above properties $(1)$-$(3)$ hold for $t,s\in \R$ instead of $t,s\in \R_+:=[0,\infty)$ we call $(T_t)_{t\in\R}$ a $C_0$-group.

The generator $(A,D(A))$ of a strongly continuous semigroup $\Tt$ on $X$ is the operator
\[
 Ax= \lim_{t\rightarrow 0} \frac{T_tx-x}{t}=\left.\frac{\partial T_tx}{\partial t}\right\vert_{t=0}
\]
defined for every $x$ in its domain
\[
 D(A)= \{x\in X: \lim_{t\rightarrow 0} \frac{T_tx-x}{t} \;\;\text{exists}\}.
\]

If $X$ is a Banach space, then the well known spectral inclusion theorem holds (\cite[2.5]{EN}). In an arbitrary locally convex space, the similar property holds for the point spectrum.

\begin{lemma}\label{fakt-mnoznikiTt}
Let $(A, D(A))$ be a generator of a strongly continuous semigroup $\Tt$ acting on a locally convex space $X$. If $x$ is an eigenvector of $A$ with eigenvalue $\lambda$ then for every $t\geq 0$ the following holds
\begin{align*}
T_tx=e^{t\lambda}x.
\end{align*}
\end{lemma} 

\begin{proof}
For a fixed eigenvector $x$ with eigenvalue $\lambda$ denote by $(S_t)_{t\geq 0}$ the rescaled semigroup $S_t=e^{-t\lambda}T_t$. 
Clearly the semigroup $(S_t)_{t\geq 0}$ is strongly continuous. We denote by $B$ the generator of $(S_t)$. For every $x\in X$ we have 
\begin{align*}
\frac{S_tx-x}{t}=\frac{e^{-\lambda t}T_tx-x}{t}
= \frac{e^{-\lambda t}T_tx-T_tx+T_tx-x}{t}
=\frac{e^{-\lambda t}-1}{t}T_tx+\frac{T_tx-x}{t}.   
\end{align*} %\xrightarrow{t\searrow 0} -\lambda x+Ax  
%As the scalar multiplication $\C\times X\rightarrow X$ is continuous the first summand converges to $\lambda x$. The second summand converges to $Ax$. The limit $\lim_{t\rightarrow 0} \frac{S_tx-x}{t}$ exists for every $x\in D(A)$. 
Since 
\begin{align*}
\frac{e^{-\lambda t}-1}{t}T_tx \xrightarrow{t\searrow 0} -\lambda x  
\end{align*}
we observe that $D(B)=D(A)$ and $B=A-\lambda$.

For $x\in D(A-\lambda)$ by (\cite[1.2]{Komura}) we have  
\begin{align*}
S_tx-x= \int_0^t S_s(A-\lambda)x \ds.
\end{align*}
Hence
\begin{align*}
e^{-\lambda t}T_tx-x= \int_0^t e^{-\lambda t}T_s(A-\lambda)x \ds
\end{align*}
As $Ax=\lambda x$ the right hand side equals $0$ and we have 
\begin{align*}
T_tx=e^{t\lambda}x.
\end{align*}
\end{proof}

It follows that

\begin{corollary}If a $C_0$-semigroup $\Tt$ is generated by a multipier $(M,(m_n))$, then it is a semigroup of multipliers. Moreover, for every $t\in\R_+$ the multipier sequence of $(T_t,(m^t_n))$ is given by $m_n^t=\exp(tm_n)$.
\end{corollary}
 
% In the paper we will need the following fact
% \begin{fact}\label{fact2}
% Let $(A, D(A))$ be a generator of a strongly continuous semigroup $\Tt$ on a webbed and ultrabornological locally convex space $X$ and let $Y$ be a dense subspace of $X$. If for some continuous operator $B$ we have $Ay=By$ for all $y\in Y$ then $A=B$ on $X$.  
% \end{fact}

We now  present some properties of the algebra of multipliers $M(\R)$.
%In the paper we will use the representation of multipliers by holomorphic functions from \cite{DL1}.
%We denote by $\hat{\C}$ the Riemann sphere and 
%by $H(\hat{\C}\setminus \R)$ the space of holomorphic functions on some open neighbourhood of $\hat{\C} \setminus \R$, i.e
% 
% %Its subspace consisting of functions vanishing at $\infty$ is denoted by $\Ho$. More precisely we consider them as the inductive limits of Fr\'echet spaces: 
%\begin{align*}
% H(\hat\C\setminus \R)=  \bigcup_{N\in\N} H(\hat\C\setminus [-N,N]) \\
% %\Ho&= \bigcup_{N\in\N} H_0(\hat\C\setminus [-N,N]).
% \end{align*}  
We denote by $\hat{\C}$ the Riemann sphere and by $\Ho$ the space of holomorphic functions around infinity, vanishing at infinity, which extend to holomorphic functions on $\hat{\C} \setminus \R$ i.e.
\[
 \Ho= \bigcup_{N\in\N} H_0(\hat\C\setminus [-N,N]).
\]

The space $\Ho$ equipped with the Hadamard multiplication of Laurent series, i.e.
\begin{align*}
 & f*g(z)=\sum_{n=0}^\infty \frac{f_ng_n}{z^{n+1}} \quad \text{around infinity} \\
\end{align*}
where
\begin{align*}
&f(z)=\sum_{n=0}^\infty \frac{f_n}{z^{n+1}}, \;\; g(z)=\sum_{n=0}^\infty \frac{g_n}{z^{n+1}} \quad \text{around infinity},
\end{align*}
forms an algebra. The algebra $\Ho$ is isomorphic to the algebra $H(\hat\C\setminus \frac{1}{\R})$ of functions holomorphic at zero which extend to holomorphic functions on $\C\setminus\R$ with Hadamard multiplication of Taylor series, i.e
\begin{align*}
 & f*g(z)=\sum_{n=0}^\infty {f_ng_n}{z^{n}} \quad \text{around zero} \\
 \intertext{where}\\
 &f(z)=\sum_{n=0}^\infty {f_n}{z^{n}}, \;\; g(z)=\sum_{n=0}^\infty {g_n}{z^{n}} \quad \text{around zero}.
\end{align*}
The isomorphism is given by the map $\phi(f)(z)=\frac{1}{z}f(\frac{1}{z})$.

To make the paper self contained we cite multiplier's representation theorem from \cite{DL1} which we will need later.

\begin{theorem}[{\cite[2.8]{DL1}}]\label{DL1}
The algebra of multipliers $M(\R)$ is topologically isomorphic as an algebra with the following algrebras of holomorphic functions:
\begin{enumerate}[$(1)$]
\item 
$\Ho$ with Hadamard multiplication of Laurent series,
% \begin{align*}
% & f*g(z)=\sum_{n=0}^\infty \frac{m_nl_n}{z^{n+1}}, \quad \text{where} \;
%  f(z)=\sum_{n=0}^\infty \frac{m_n}{z^{n+1}}, \; g(z)=\sum_{n=0}^\infty \frac{l_n}{z^{n+1}}
% \end{align*}
%The multiplier sequence of the given multiplier is equal to the Laurent coefficients $(m_n)$ of the corresponding function $f$.
\item
$H(\hat\C\setminus \frac{1}{\R})$ with Hadamard multiplication of Taylor series.
% \begin{align*}
% & f*g(z)=\sum_{n=0}^\infty {m_nl_n}{z^{n}}, \quad \text{where} \;
%  f(z)=\sum_{n=0}^\infty {m_n}{z^{n}}, \; g(z)=\sum_{n=0}^\infty {l_n}{z^{n}}
% \end{align*}

\end{enumerate}
The multiplier sequence of the given multiplier is equal to the Laurent (Taylor) coefficients at infinity (zero) $(f_n)$ of the corresponding function $f$.

\end{theorem}

\section{Main results}

Now we present the theorem which will be our main tool in proving that some multipliers do or do not generate $C_0$-semigroups.

\begin{theorem}\label{Ttzwarty}
Let $\map{M}{\A(\R)}{\A(\R)}$ be a multiplier with the multiplier sequence $(m_n)_{n\in\N}$. The following assertions are equivalent:
\begin{enumerate}[$(i)$]
 \item The multiplier $M$ generates a $C_0$-semigroup $\Tt$.
 \item For every $t\in \R_+$ the operator $T_t$ is a multiplier with the multiplier sequence $(m^t_n)_{n\in\N}=(\exp(tm_n))_{n\in\N}$ and the map $\map{Tf}{\R_+}{\Ar}$, $Tf(t)=T_tf$ is continuous for every $f\in\Ar$.
 \item For every $t\in \R_+$ the operator $T_t$ is a multiplier with the multiplier sequence $(m^t_n)_{n\in\N}=(\exp(tm_n))_{n\in\N}$ and the set $\{ T_tf:\; t\in [0,t_0]\}$ is bounded in $\A(\R)$ for every $f\in \Ar$ and every $t_0\geq 0$.
\end{enumerate}

%The multiplier $(M,(m_n))$ generates a semigroup $\Tt$ if and only if for every $t\in \R_+$ the operator $T_t$ is a multiplier with the multiplier sequence $(m^t_n)_{n\in\N}=(\exp(tm_n))_{n\in\N}$.
%
%Moreover, for $(T_t, (m_n^t))$, the family $\Tt$ is a $C_0$-semigroup if and only if the set $\{ T_tf:\; t\in [0,t_0]\}$ is bounded in $\A(\R)$ for every $f\in \Ar$ and every $t_0\geq 0$
%The multiplier $M$ generates a $C_0$-semigroup $\Tt$ if and only if for every $t\in \R_+$ the operator $T_t$ is a multiplier with the multiplier sequence $(m^t_n)_{n\in\N}=(\exp(tm_n))_{n\in\N}$ and the set $\{ T_tf:\; t\in [0,t_0]\}$ is bounded in $\A(\R)$ for every $f\in \Ar$ and every $t_0\geq 0$
\end{theorem}

\begin{proof}
$(i)\Rightarrow(ii)$: 
Follows from Fact \ref{fakt-mnoznikiTt}.

$(ii)\Rightarrow(iii)$: Obvious.

$(iii)\Rightarrow (i)$:
First we will show that multipliers $(T_t,(m^t_n))$ form a semigroup. For every $t,s \geq 0$ and every monomial $x^n$ we have
\begin{align*}
T_tT_s x^n=T_t e^{sm_n}x^n=e^{(t+s)m_n}x^n= T_{t+s}x^n.
\end{align*}
Since polynomials are dense in $\Ar$ we get that $T_tT_s = T_{t+s}$ for every $t,s\geq 0$ and $\Tt$ is indeed a semigroup.

Now we will show that $\Tt$ is a $C_0$-semigroup. We assume that the set $\{ T_tf:\; t\in [0,t_0]\}$ is bounded in $\Ar$ for arbitrary $f\in\Ar$, $t_0\geq 0$. By $\tau$ we denote the natural topology on $\Ar$.

Recall that an operator $V:\Ar=\indH \rightarrow \C$ is continuous if and only if $\map{V}{H(U)}{\C}$ is continuous for every complex neighbourhood $U\supset \R$ (\cite[1.25]{Dnotes}).
The linear map
\begin{align*}
 B:\Ar &\longrightarrow \omega \\
 f &\longmapsto\left(\frac{f^{(n)}(0)}{n!}\right)_n,
\end{align*}
is continuous since for the topology of pointwise convergence $\tau_\omega$ on $\omega$ and from the Cauchy inequality we get 
\begin{align*}
 \abs{\frac{f^{(n)}(0)}{n!}} \leq C_K \norm{f}_{\infty, K}
\end{align*}
for any compact set $K\subset U$ with $0\in \Int K$. Hence we can consider $\Ar$ with the coarser topology induced by the map above i.e. $\tau_2=B^{-1}(\tau_\omega)$.

The multiplier sequence of $T_t$ equals $(e^{tm_n})_{n\in\N}$. Hence $(T_tf)^{(n)}(0)=e^{tm_n}f^{(n)}(0)$ and the map 
\begin{align*}
C_f: \R_{+} &\longrightarrow \omega \\
 t &\longmapsto \left(\frac{(T_tf)^{(n)}(0)}{n!}\right)_n= \left(\frac{e^{tm_n} f^{(n)}(0)}{n!}\right)_n
\end{align*}
is continuous. 

We consider the mapping $\map{Tf}{\R_+}{(\Ar, \tau)}$, $Tf(t):=T_tf$. The map $\map{Tf}{\R_+}({\Ar, B^{-1}(\tau_\omega)})$ is continuous. Indeed, take an open set $U\in B^{-1}(\tau_\omega)$. Hence, there exists an open set $V\in\omega$ such that $U=B^{-1}(V)$ and we have $(Tf)^{-1}(U)=(Tf)^{-1}(B^{-1}(V))=(B\circ Tf)^{-1}(V)=C_f^{-1}(V)$.

Since by the assumption the set $\{T_tf: \; t\in[0,t_0]\}$ is bounded in $(\Ar, \tau)$, hence compact and the compact Hausdorff topology is the minimal Hausdorff topology \cite[3.1.14]{Engelking} we get that $\tau=\tau_2$ on $\{T_tf: t\in[0,t_0]\}$ and the map $\map{Tf}{[0,t_0]}{(\Ar, \tau)}$ is continuous for every $t_0\geq 0$. Hence $\Tt$ is strongly continuous.

Denote by $A$ the generator of the semigroup $\Tt$. For every monomial $x^n$ we have
\begin{align*}
Ax^n=\lim_{t\searrow 0}\frac{T_tx^n-x^n}{t}=\lim_{t\searrow 0} \frac{e^{tm_n}x^n-x^n}{t}
= \lim_{t\searrow 0} \frac{e^{tm_n}-1}{t}x^n=m_nx^n.
\end{align*} 
Hence, $A=M$ on the set of polynomials, which is dense in $\Ar$. As the operator $M$ is continuous, for any function $f\in \Ar$ and a sequence of polynomials $p_n$ converging to $f$, we have  $Ap_n=Mp_n\rightarrow Mf$ in $\Ar$. Because the generator $A$ is closed \cite[1.4]{Komura} we get that $f\in D(A)$ and $Af=Mf$.

\end{proof}

The above with Theorem \ref{DL1} gives

\begin{corollary} \label{wniosek}
The following assertions are equivalent
\begin{enumerate}[$(1)$]
	\item	
	The multiplier $(M,(m_n))$ generates a $C_0$-semigroup $\Tt$ on $\Ar$
	\item
	For every $t\geq 0$ the function $f_t$, $f_t(z)=\sum_{n=0}^\infty \exp(tm_n)z^n$, extends to a holomorphic function 
belonging to $H(\hat\C\setminus \frac{1}{\R})$ and the set $\{f_t: t \leq t_0\}$ is bounded in $ H(\hat\C\setminus \frac{1}{\R})$ for all $t_0\geq 0$.
	\item 
	For every $t\geq 0$ the function $\widetilde{f}_t$, $\widetilde{f}_t(z)=\sum_{n=0}^\infty \frac{\exp(tm_n)}{z^{n+1}}$, extends to a holomorphic function 
belonging to $\Ho$ and the set $\{\widetilde{f}_t: t \leq t_0\}$ is bounded in $\Ho$ for all $t_0\geq 0$.
\end{enumerate}

\end{corollary}
% Hence, the multipier $M, (m_n)$ generates a $c_0$ - semigroup if and only if for all $t\geq0$ the corresponding functions $f_t(z)=\sum_{n=0}^\infty \exp(tm_n)z^n \in H(\C\setminus \R)$ and $f_t(z)=\sum_{n=0}^\infty \frac{\exp(tm_n)}{z^{n+1}} \in \Ho$ and the map $\map{T(\cdot)}{\R_+}{\A(\R)}$, $T(t)=T_tf$ is continuous for every $f\in\A(\R)$.

\begin{proof}
$(1)\Leftrightarrow(2)$: 
By Theorem \ref{Ttzwarty} statement $(1)$ is equivalent to $T_t$ being multipliers with multiplier sequences $(e^{tm_n})_{n\in\N}$ and $\{T_tf: t\leq t_0\}$ being bounded in $\Ar$ for all $t_0>0$ and all $f\in\Ar$. The first condition by Theorem \ref{DL1} is equivalent to $f_t\in H(\hat\C\setminus \frac{1}{\R})$ for all $t\geq 0$. In view of the uniform boundness principle the second condition is equivalent to $\{T_t: t\leq t_0\}$ being bounded in $\LL(\Ar)$, which by Theorem \ref{DL1} is equivalent to $\{f_t: t\leq t_0\}$ being bounded in $H(\hat\C\setminus \frac{1}{\R})$. 

$(1)\Leftrightarrow(3)$: the proof of the equivalence is similar to the above.
\end{proof}

\begin{lemma}\label{additive}
 The set of multipliers generating a $C_0$-semigroup is additive.
\end{lemma}

\begin{proof}
 Let multipliers $(A, (a_n))$, $(B,(b_n))$ be the generators of $C_0$-semigroups $(T_t^A,(e^{ta_n}))_{t\geq 0}$ and $(T^B_t,(e^{tb_n}))_{t\geq 0}$ respectively and let $f_t,g_t\in H(\hat\C\setminus \frac{1}{\R})$ be the corresponding (in view of Theorem \ref{DL1}) holomorphic functions. Take $t\geq 0$ and choose $0<\eps,\delta<1$ such that $f_t\in H( \hat\C\setminus ((-\infty, -\eps]\cup [\eps,\infty)))$ and $g_t\in H( \hat\C\setminus ((-\infty, -\delta]\cup [\delta,\infty)))$. By the Hadamard multiplication theorem $f_t\ast g_t\in H( \hat\C\setminus ((-\infty, -\eps\delta]\cup [\eps\delta,\infty)))$ \cite[Th. H]{Muller}. Hence by Theorem \ref{DL1} the operator $T^{A+B}_t$ is a multiplier with a mutliplier sequence $(e^{t(a_n+b_n)})_{n\geq 0}$. Since for monomials we have $T_t^{A+B}x^n=e^{t(a_n+b_n)}x^n=T_t^AT_t^Bx^n$ and momomials are linearly dense in $\Ar$, we get that $T_t^{A+B}=T_t^AT_T^B$. Hence the map $\map{T^{A+B}f}{\R_+}{\Ar}$, $T^{A+B}f(t)=T_t^{A+B}f$ is continuous for all $f\in\Ar$. Thus by Theorem \ref{Ttzwarty} the multiplier $(A+B, (a_n+b_n))$ generates a $C_0$-semigroup $(T_t^{A+B})_{t\geq 0}$.

 \end{proof}

Now we answer the question when does the Euler differential operator generate a strongly continuous semigroup.

\begin{theorem}\label{euler1}
 Let $E\in L(\Ar)$ be a first order Euler differential operator, 
 \[
  Ef(x)=axf'(x)+bf(x).
 \]
The multiplier $E$ generates a $C_0$-semigroup if and only if $a\in\R$.
\end{theorem}

\begin{proof}
A multiplier $(M,(c))$ with a constant multiplier sequence generates the $C_0$-semigroup $(T_t)_{t\geq 0}$, $T_tf=e^{ct}f$. Hence by Lemma \ref{additive} without loss of generality we can assume that $b=0$.  

The multipier sequence of $E$ is $(m_n)=(an)$. Hence we get the corresponding functions
\begin{align}\label{deg1}
 {f}_t(z)=\sum_{n=0}^\infty e^{tan}{z^{n}}= \frac{1}{1-ze^{ta}} \in H_0(\widehat\C\setminus e^{-ta}).
\end{align}

Hence ${f}_t\in H(\hat\C\setminus \frac{1}{\R})$ for every $a \in \R$, $t\geq 0$, and $(T_t,(e^{tan})) $ is a multiplier. On the other hand, if $a\notin \R$ then for every $t$ such that $ta\neq k\pi i$, $k\in\Z$, we have ${f}_t\notin H(\hat\C\setminus \frac{1}{\R})$ and $E$ does not generate a semigroup.

To finish the proof we need to show that, under the assumption $a\in\R$, the semigroup $\Tt$ is strongly continuous, i.e. we need to prove the continuity of the map $\map{Tf}{\R_+}{\A(\R)}$, $Tf(t)=T_f(t)$ for arbitrary $f\in\A(\R)$.
By (\ref{deg1}) we can extend the map $\map{Tf}{\R_+}{\A(\R)}$ to the map $\map{Tf}{\R}{\A(\R)}$.

To prove the continuity we will use the explicit formula of the multipliers $T_t$ with $(m_n^t)=(\exp(tan))$. We have $T_tf(x)=f(e^{ta}x)$.
Indeed, for a monomial $x^n$ we have 
\begin{align*}
T_t x^n(y)= e^{tan} x^n(y)= e^{tan}y^n=x^n(e^{ta}y) .
\end{align*}
Moreover, observe that the map $f\mapsto g$, $g(x)=f(e^{ta}x)$ is linear and continuous on $\Ar$ for any $a,t\in\R$.
Thus the claim follows from the density of polynomials in $\Ar$.

As $T_tf-T_{t+s}f=T_t(f-T_sf)$ and $s\in\R$ it is enough to show the continuity at $t=0$.
Recall that $T_{t_n}f \rightarrow f$ in $\Ar$ as $t_n\rightarrow 0$ if and only if there exists an open complex neighbourhood $U\supset\R$ such that $T_{t_n}f\in H(U)$ for every $n\in\N$ and $T_{t_n}f\rightarrow f$ in $H(U)$.

Let $U$ be a complex open neighbourhood of $\R$ such that $f\in H(U)$. Let $U'$ be a star-convex subset of $U$ and put $V:=\frac{1}{2}U'$. We choose $\eps>0$ such that $e^{\abs{a}\eps}<2$. Then for $\abs{t}<\eps$ we have that $e^{t a}V\subset U'\subset U$ and $T_tf\in H(V)$.  

Now we will show that $T_{t_n}\rightarrow f$ in $H(V)$. Take an arbitrary compact set $K\subset V$. Then for a compact set $K_2$ such that $K\subset K_2 \subset V$, $K \subset \Int K_2$ and for ${t_n}$ small enough we have $e^{t_na}K\subset K_2\subset V$ and 
\begin{align*}
 \lim_{t_n\rightarrow 0}\norm{T_{t_n}f-f}_K= \lim_{t_n\rightarrow 0} \sup_{z\in K} \abs{f(e^{t_na}z)-f(z)}=0,
\end{align*}
since $f$ is uniformly continuous on compact sets.

We have proved that $\Tt$ is strongly continuous. Moreover $E$ is its generator as for all monomials we have
\[
 \lim_{t\rightarrow 0} \frac{T_tx^n-x^n}{t}=  \lim_{t\rightarrow 0} \frac{(e^{ta}x)^n-x^n}{t}= \lim_{t\rightarrow 0} \frac{e^{atn}-1}{t}x^n= anx^n=Ex^n.
\]

\end{proof}

Now we consider the differential operators $P(\theta)$ of higher orders. We start with the negative result.

\begin{theorem}
 Let $P(\theta)=\sum_{k=0}^Ka_k\theta^k$, $\theta f(x)=xf'(x)$, be a finite order differential operator of degree at least $2$. The operator 
 $P(\theta)$ does not generate a $C_0$-semigroup in the following cases:
 \begin{enumerate}[$(1)$]
 \item \label{N1}
 $\Rea a_K=\ldots =\Rea a_{l+1}=0$ and $\Rea a_l>0$ for some $l\geq 2$.
  \item \label{NiQ}
  $a_K,\ldots, a_2\in i\Q$.
 \end{enumerate}

\end{theorem}

\begin{proof}
$(\ref{N1})$: The multiplier sequence of $P(\theta)$ is given by $(m_n)=(P(n))$. Assume that $P(\theta)$ generates a $C_0$-semigroup $\Tt$. Then, by Corollary \ref{wniosek}, for all $t\geq 0$ the operator $(T_t, e^{tP(n)})$ is a multiplier and the function $f_t$, $f_t(z)=\sum_{n=0}^\infty e^t{P(n)}z^n$ around $0$, extends to a holomorphic function in $H(\hat{\C}\setminus \R)$. But, for every $R>0$ we have
  \[
   \sup_{n\in\N} \abs{e^{tP(n)}}R^n=\sup_{n\in\N} e^{t\Rea P(n)}R^n> \sup_{n\in\N} e^{t(a_l-\eps)n^l}R^n =\infty
  \]
for some $\eps>0$.
  
$(\ref{NiQ})$: We start with the case $P(\theta)=\sum_{k=1}^K{ {a}_k \theta^k}$ such that $a_k\in iQ$ for every $1\leq k\leq K$. We will show, that for every such polynomial $P$ there exists $t_0\in \R_+$ such that $(m^{t_0}_n)_{n\in\N}=(\exp(t_0P(n)))_{n\in\N}$ is not a multiplier sequence.  

Let $\f{P}(x)=\sum_{k=1}^K{ \f{a}_k x^k}$ be a polynomial such that $\f{a}_k\in \Z$ for all $k\leq K$ and 
$m_n=\frac{i}{S}\f{P}(n)$, where $S$ is the common denominator of all the coefficients $a_k$.

As $\f{a_0}=0$ we have that $\f{P}(0)=0$. Let $n_0\in \N$ be such that 
\begin{enumerate}
 \item $\abs{\f{P}(n_0+2)}= q$,  $q>2$,   %$2|q$ i
% \item $\abs{\f{P}(n_0)}>n_0$,
 \item $\f{P}(n_0)\not\equiv \f{P}(n_0+2)$ mod $2q$.
\end{enumerate}
It is clear that such $n_0$ exists.  Indeed, take $n_0$ such that $P(n)$ is monotonous for $n\geq n_0$. Then $\abs{\f{P}(n_0)}<\abs{\f{P}(n_0+2)}<2q$.

Take $t_0= \frac{S\pi}{q}$ and consider the function 

\begin{align*}
f_{t_0}(z)=\sum_{n=0}^\infty {m^{t_0}_n} z^n = \sum_{n=0}^\infty \exp\left(\frac{\f{P}(n)}{q}\pi i\right) z^n \quad\; \text{around $0$.}
\end{align*}
 The expression $\exp\left(\frac{\f{P}(n)}{q}\pi i\right)$ takes at most $2q$ different values and
\begin{align*}
 \exp\left(\frac{\f{P}(n)}{q}\pi i\right)=\exp\left(\frac{\f{P}(2q+n)}{q}\pi i\right). 
\end{align*}
% Moreover we have
%  \begin{align*}
% \exp\left(\frac{\f{P}(0)}{q}\pi i\right)= 1 = - \exp\left(\frac{\f{P}(n_0)}{q}\pi i\right).
%\end{align*}
Denote $\xi_n=\exp\left(\frac{\f{P}(n)}{q}\pi i\right)$. Hence we have
\begin{align*}
 f_{t_0}(z)&= \sum_{n=0}^\infty \xi_n z^n = z^0+ \xi_1 z^1 + \xi_2 z^2 + \ldots + z^{2q} + \xi_1z^{2q+1} + \xi_2z^{2q+2}+\ldots  \\
 &=\frac{\sum_{n=0}^{2q-1} \xi_n z^n }{1-z^{2q}}
\end{align*}

This implies that $f$ is defined on $\C$ except it can have poles of order $1$ at $2q$-roots of unity.
Now we will show that $f_{t_0}\notin H(\C\setminus \frac{1}{\R})$. 
Assume that $f_{t_0} \in H(\C\setminus \frac{1}{\R})$, so $f_{t_0}$ would have only poles of order $1$ in points $\pm 1$. Then $g(z)=(1-z^2)f(z)\in H(\C)$. But
\begin{align*}
 g(z)&= (1-z^2)f_t(z)
 = (1-z^2)\sum_{n=0}^\infty \xi_n z^n
 = \sum_{n=0}^\infty (\xi_n z^n- \xi_n z^{n+2})\\
 &= 1 +\xi_1 z + \sum_{n=2}^\infty (\xi_n-\xi_{n-2})z^n.
\end{align*}
For every $k\in\N$ we have 
\begin{align*}
 \xi_{2kq+n_0+2}= \xi_{n_0+2}\neq \xi_{n_0} =\xi_{2kq+n_0}
\end{align*}
and 
\begin{align*}
 \abs{\xi_{2kq+n_0+2}-\xi_{2kq+n_0}}=\delta 
\end{align*}
for some $\delta>0$. 
Hence  
\begin{align*}
 \limsup_{n\rightarrow \infty}\sqrt[\leftroot{-3}\uproot{3}n]{\abs{\xi_n-\xi_{n-2}}}=1
\end{align*}
and we get a contradiction. Hence $f_{t_0}\notin H(\C\setminus \frac{1}{\R})$ and $(P(\theta),(P(n)))$ does not generate a semigroup. 
  
Now consider $P(\theta)= \sum_{k=1}^K a_k\theta^k$, with $a_K,\ldots, a_2\in i\Q$, $a_1=ir$, $r\in \R\setminus \Q$. Taking $t_0=2S\pi$, where $S$ denotes the common denominator of $a_K, \ldots a_2$, we get that 
\begin{align*}
 e^{t_0P(n)}= e^{2r\pi i n}.
\end{align*}
By Theorem \ref{DL1} the operator $(T_{t_0}, ( e^{2r\pi i n}))$ is not a multiplier since 
\begin{align*}
f_{t_0}(z)= \sum_{n=0}^\infty e^{2r\pi i n}z^n= \frac{1}{1-e^{2r\pi i}z} \notin H(\hat\C\setminus \frac{1}{\R})
\end{align*}
as $r\notin \Q$. By Theorem \ref{fakt-mnoznikiTt}, $(P(\theta),P(n))$ cannot generate a semigroup.

Summarizing, we proved that multiplier $(P(\theta),(P(n)))$ with $P(\theta)=\sum_{k=1}^K a_k\theta^k$, $a_K, \ldots a_2 \in i\Q$, $a_1\in \C\setminus\R$ does not generate a semigroup. Now take a multiplier $Q(\theta)= P(\theta)+b_1\theta +c$ with $b_1\in \R$. As the operators $(M_{-b}, (-b_1n-c))$, $(M_b, (b_1n+c))$ generate $C_0$-semigroups (Theorem \ref{euler1}) and the sum of multipliers being generators  is a generator (Lemma \ref{additive}) we conlude that $(Q(\theta), (Q(n)))$ generates the semigroup if and only if $(P(\theta), (P(n)))$ does, which finishes the proof.
\end{proof}

%\todo{problem} 

Now we will give another example of a multipier that generates a strongly continuous semigroup on $\Ar$, i.e. we will show that the Hardy operator, $Hf(x)=\frac{1}{x}\int_0^x f(t)dt$, is a generator of a $C_0$-group. To do this we need some more facts from the theory of the space of analytic functions. In particular, we  need a representation of multipiers by the so called Mellin functions. Hence, we start with following definitions.

\begin{definition}
Let $(\kappa_n)_{n\in\N}$, $(K_n)_{n\in\N}$ be sequences of real numbers such that $\kappa_1<0$ and $0<K_n\rightarrow \infty$.
We define an asymptotic halfplane $\omega$ by
\begin{align*}
\omega=\bigcup_{n=1}^\infty (\kappa_n+\omega_{K_n}) \; \text{for $\omega_{K_n}:=\{z\in\C: \abs{\Ima z}<K_n \Rea z\}$}. 
\end{align*}

% 
% Let $\omega$ denote an asymptotic halfplane, i.e. for some $\kappa_n$, $\kappa_1<0$ and $0<K_n\rightarrow \infty$ we define
% \begin{align*}
% \omega=\bigcup_{n=1}^\infty (\kappa_n+\omega_{K_n}) \; \text{for $\omega_{K_n}:=\{z\in\C: \abs{\Ima z}<K_n \Rea z\}$}. 
% \end{align*}
We call a holomorphic function $f\in H(\omega)$ a Mellin function for the sequence $(m_n)_{n\in \N}$ if there exists a constant $C>0$ such that
\begin{align*}
 \abs{f(z)}\leq Ce^{C\abs{\Rea z}} \; \text{for $z\in\omega$}
\end{align*}
and
\begin{align*}
 f(n)=m_n.
\end{align*}
We will denote the space of Mellin functions by $\mH(\omega)$. 
\end{definition}

\begin{definition}
 For $a\in R$ we define
 \begin{align*}
  \mH_a(\omega)= \{f\in \mH(\omega): \forall j \sup_{z\in \Gamma_j}\abs{f(z)}e^{-(a+\frac{1}{j})\Rea z} <\infty\}
 \end{align*}
where $\Gamma_j=\overline{(\cup_{n\leq j}(\kappa_n+1/j+ \omega_{K_n})}$.
\end{definition}

The space $\mH_a(\omega)$ is a Fr\'echet space with the fundamental system of seminorms $(\norm{\cdot}_j)_{j\in \N}$ given by
\[
 \norm{f}_j=\sup_{z\in \Gamma_j} \abs{f(z)} e^{-(a+\frac{1}{j})\Rea z}.
\]

We will need the following theorems.
\begin{theorem}[{\cite[4.1]{DL}}]\label{H+}
 There exists a continuous, linear and surjective mapping $\map{H_a^+}{\mH_a(\omega)}{\A([0,e^a])'}$ satisfying
 \begin{align*}
  \langle H_a^+(f),x^n \rangle=f(n) \; \text{for every $n\in\N$}.
 \end{align*}
\end{theorem}

\begin{theorem}[{\cite[2.6]{DL1}}]\label{B}
The map 
\[
 \map{\B}{\Ar_b'}	{M(\R)},\quad  \B (F) g(y)=\langle g(y\cdot), F\rangle
\]
is a linear homeomorphism and the multiplier sequence of $\B(F)$ is equal to the sequence of
moments of the analytic functional $F$, i.e. to $(\langle z^n,F\rangle)_{n\in\N}$.
\end{theorem}

We can now prove the following theorem

\begin{theorem}
Let $H\in \LL(\Ar)$ be the Hardy operator, $Hf(x)=\frac{1}{x}\int_0^x f(y) \dy$. The operator $A=\sum_{k=0}^K a_kH^k$, $a_1,\ldots, a_K\in\C$ generates a $C_0$-group on $\Ar$.
\end{theorem}

\begin{proof}
 The multiplier sequence of the Hardy operator $H$ equals $(\frac{1}{n+1})_{n\in\N}$. Hence the multiplier sequence of $(A,(m_n))$ equals $m_n=\sum_{k=0}^K \frac{a_k}{(n+1)^k}$. We will use Theorem \ref{Ttzwarty}, hence it is enough to show that sequences $\left(\exp\left(\sum_{k=0}^K \frac{ta_k}{(n+1)^k}\right)\right)_{n\in\N}$ are multiplier sequences for multipliers $T_t$ and that the mapping $\map{Tf}{\R}{\Ar}$, $Tf(t)=T_tf$ is continuous. From Theorem (\ref{B}) the first condition is equivalent to the existence of functionals $F_t\in\Ar'$ satisfying $\langle F_t,x^n \rangle=\sum_{k=0}^K \frac{a_k}{(n+1)^k}$, and due to Theorem \ref{H+} it is equivalent to the existence of the Mellin functions $\mu_t\in \mH_a(\omega)$ for $\left(\exp\left(\sum_{k=0}^K \frac{a_k}{(n+1)^k}\right)\right)_{n\in\N}$.

For the proof it is enough to find the asymptotic halfplane $\omega$ and Mellin functions $\mu_t\in \mH_a(\omega)$ such that the mapping $\map{\phi}{\R}{\mH_a(\omega)}$, $t\mapsto \mu_t$ is continuous. Indeed, consider the following diagram
\begin{align*}
\R \xrightarrow{\phi} \mH_a(\omega) \xrightarrow{H^+_a} \A([0,e^a])' \xrightarrow{\B} M(\R).
\end{align*}
Recall that $H^+$, $\B$ are continuous (Theorems \ref{H+}, \ref{B}) with $\B \circ H^+ \circ \phi (t)= T_t$.
Hence, if the function $\phi$ is continuous then the function $t\mapsto T_tf$ is continuous.

Let $\omega$ be an asymptotic halfplane such that $\kappa_1=-\frac{1}{2}$, $\kappa_n=0$ for all $n\geq 2$ and consider the functions $\mu_t= \exp\left(\sum_{k=0}^K \frac{ta_k}{(z+1)^k}\right)$, $t\geq 0$. 

Then $\mu_t$ is clearly holomorphic on $\omega$ and for $z\in \omega\subset \{\Rea z>-\frac{1}{2}\}$ it satisfies
\begin{align*}
 \abs{\mu_t(z)}&=\abs{\exp\left(\sum_{k=0}^K \frac{ta_k}{(z+1)^k}\right)}\leq \exp\left(\sum_{k=0}^K \abs{\frac{ta_k}{(z+1)^k}}\right) \leq \exp\left(\sum_{k=0}^K 2^k\abs{t{a_k}}\right)\\
 &< \exp\left(\sum_{k=0}^K 2^k\abs{{ta_k}} +\frac{1}{2}\right)\exp(\Rea z).
\end{align*}
Hence $\{\mu_t\}_{t\geq 0}\subset \mH(\omega)$ and because
\begin{align*}
 \mu_t(n)=\exp\left(\sum_{k=0}^K \frac{ta_k}{(n+1)^k}\right) ,
\end{align*}
we get that functions $\mu_t$ are Mellin functions for the sequence $\left(\exp\left(\sum_{k=0}^K \frac{ta_k}{(n+1)^k}\right)\right)_{n\in \N}$.

Now we will show that $\mu_t\in \mH_a(\omega)$ for any $a>0$ and all $t\in \R$. We compute

\begin{align*}
\sup_{z\in \Gamma_j}\abs{\mu_t(z)}e^{-(a+\frac{1}{j})\Rea z}
&=\sup_{z\in \Gamma_j}   \abs{\exp\left(\sum_{k=0}^K \frac{ta_k}{(z+1)^k}\right)}\exp\left({-\left(a+\frac{1}{j}\right)\Rea z}\right)\\
&\leq \exp\left(\sum_{k=0}^K 2^k\abs{t{a_k}}\right) \sup_{z\in \Gamma_j}\exp\left({-\left(a+\frac{1}{j}\right)\Rea z}\right)\\
&<\exp\left(\sum_{k=0}^K 2^k\abs{t{a_k}}\right) \exp\left({\left(a+\frac{1}{j}\right)\frac{1}{2}}\right) <\infty.
\end{align*}

To finish the proof we need to prove the continuity of the map $\map{\phi}{\R}{\mH_a(\omega)}$, $\phi(t)=\mu_t$. 

Fix $t\in\R$, $j\geq 1$. Then
\begin{align*}
\norm{\mu_t-\mu_{t+h}}_j
&= \sup_{z\in \Gamma_j}\abs{\mu_t(z)-\mu_{t+h}(z)}\exp\left(-\left(a+\frac{1}{j}\right)\Rea z\right) \\
&= \sup_{z\in \Gamma_j}\abs{\mu_t(z)}\abs{(1-\mu_h(z))}\exp\left(-\left(a+\frac{1}{j}\right)\Rea z\right)\\
&<\exp\left(\sum_{k=0}^K 2^k\abs{t{a_k}}+\frac{1}{2}(a+j^{-1})\right)\sup_{z\in \Gamma_j}\abs{(1-\mu_h(z))}.
\end{align*}
For the last component we have that
\begin{align}\label{prod}
\sup_{z\in \Gamma_j}\abs{(1-\mu_h(z))} = \sup_{z\in \Gamma_j}\abs{(1-\prod_{k=0}^K\exp\left(\Rea({a_k})h\frac{\Rea \overline{(z+1)}^k}{\abs{z+1}^k}\right)\exp\left(i\Ima({a_k})h\frac{\Ima \overline{(z+1)}^k}{\abs{z+1}^k}\right)}.
\end{align}
Since for all complex numbers $z\in \C$  
\begin{align*}
\frac{\Rea z}{\abs{z}}\leq 1 \quad \text{and} \quad \frac{\Ima z}{\abs{z}}\leq 1 
\end{align*}
all components of the product in (\ref{prod}) tend to $1$ uniformly on $\Gamma_j$ as $h$ tends to $0$. Hence
\[
 \norm{\mu_t-\mu_{t+h}}_j \xrightarrow{{h}\rightarrow 0} 0.
\]

\end{proof}

\end{document}